\documentclass[11pt,a4paper]{article}
\usepackage{latexsym}
\usepackage{amssymb}
\usepackage{amsthm}
\usepackage{xcolor}
\usepackage{multicol}
\usepackage{amsmath}
\title{Weak solutions to the quaternionic Monge-Amp\`ere equation}
\author{Marcin Sroka}
\date{}

\usepackage{geometry}
\newtheorem{theorem}{Theorem}
\newtheorem*{theorem*}{Theorem}
\newtheorem{lemma}{Lemma}

\newtheorem{proposition}{Proposition}

\newtheorem*{definition*}{Definition}
\newtheorem{remark}{Remark}
\newcommand{\hh}{\mathbb{H}}
\newcommand{\rr}{\mathbb{R}}
\newcommand{\cc}{\mathbb{C}}
\newcommand{\hn}{\mathbb{H}^n}

\newcommand{\cn}{\mathbb{C}^n}
\newcommand{\ii}{\mathfrak{i}}
\newcommand{\jj}{\mathfrak{j}}
\newcommand{\kk}{\mathfrak{k}}

\newcommand{\bpar}{\overline{\partial}}
\newcommand{\jpar}{\partial_J}
\newcommand{\psh}{\mathcal{PSH}}
\newcommand{\qpsh}{\mathcal{QPSH}}
\newcommand{\leb}{\mathcal{L}}
\begin{document}
\maketitle
\textbf{Abstract:} We solve the Dirichlet problem for the quaternionic Monge-Amp\`ere equation with a continuous boundary data and the right hand side in $L^p$ for $p>2$. This is the optimal bound on $p$. We prove also that the local integrability exponent of quaternionic plurisubharmonic functions is two which turns out to be less than an integrability exponent of the fundamental solution. 



\section{Introduction}
Pluripotential theory, initiated in the seminal papers of Bedford and Taylor \cite{BT76, BT82}, has become a powerful tool for solving problems in complex analysis and geometry. It has been generalized in many directions in the last decade. The most general setting are calibrated geometries, this was extensively studied in a long series of papers by Harvey and Lawson, cf. \cite{HL09b}. Even before that the basics of pluripotential theory in $\hn$ were recreated by Alesker, cf. \cite{A03a}, and more generally on hypercomplex manifolds by Alesker and Verbitsky, cf. \cite{AV06}. In this paper we wish to concentrate on the flat space $\hn$. 

The short historical overview is as follows. Quaternionic plurisubharmonic functions in $\hn$ and their basic properties were investigated in \cite{A03a}. Inspired by \cite{BT76} Alesker developed there the foundations of pluripotential theory in the quaternionic setting showing among other things that a quaternionic Monge-Amp\`ere operator defined for smooth functions as the Moore determinant, cf. \cite{M22}, of a quaternionic Hessian can be extended to the class of continuous functions. In \cite{A03b} he solved the Dirichlet problem in a quaternionic strictly pseudoconvex domain $\Omega \subset \hn$ with a continuous boundary data and the Monge-Amp\`ere mass continuous up to the boundary. Only recently Wan, cf. \cite{W18}, obtained another results in this direction. Following the approach of Ko\l odziej from \cite{K95, K05} she proved that the Dirichlet problem admits a bounded solution provided the right hand side is a finite Borel measure and a subsolution to the problem exists. Motivated by reasoning presented in \cite{CP92} and using comparison of real and quaternionic Monge-Amp\`ere operators she showed existence of continuous solutions to the Dirichlet problem for densities in $L^q$, $q \geq 4$. To sum up the strongest known result concerning existence of a continuous solution to the Dirichlet problem with a degenerate right hand side is as follows
\begin{theorem*} Suppose $\Omega \subset \hn$ is a quaternionic strictly pseudoconvex domain and $f \in L^q(\Omega)$ for $q \geq 4$ is a non negative function. Then the Dirichlet problem \begin{center} $\begin{cases} u \in \qpsh(\Omega) \cap C(\overline{\Omega}) \\ (\partial \jpar u)^n = f \Omega_{n} \\ u_{|\partial \Omega}=\phi \in C(\partial \Omega) \end{cases}$ \end{center} has a unique solution. \end{theorem*}
The regularity of solutions (except for a ball which was discussed earlier by Alesker in \cite{A03b}) was proven by Zhu, cf. \cite{Z17}. More precisely using the ideas presented in \cite{CKNS85} he proved the following result 
\begin{theorem*} For a quaternionic strictly pseudoconvex domain $\Omega \subset \hn$, $f \in C^\infty(\overline{\Omega}\times \rr)$ a positive function such that $f_x$ is nonnegative on $\overline{\Omega}\times \rr$ and $\phi \in C^\infty(\partial \Omega)$ the Dirichlet problem  \begin{center} $\begin{cases} u \in \qpsh(\Omega) \cap C^\infty(\overline{\Omega}) \\ \det (\frac{\partial^2 
{u}(q)}{\partial \bar{q_\alpha}\partial q_\beta})_{\alpha, \beta \in \{1,...,n \}} = f(q,u(q)) \:  in \: \Omega \\ {u}_{|\partial \Omega}= \phi \end{cases}$ \end{center} has a unique smooth solution.\end{theorem*} In the meantime  quaternionic pluripotential theory was further developed in \cite{WZ15, WK17, WW17} of what we will make an extensive use. Contents of those papers will be discussed below in more details.   
For results concerning Dirichlet problems in this more general approach of Harvey and Lawson one can consult \cite{HL09a} for the flat case, \cite{HL11} for manifolds and \cite{HL18} for a degenerate case. 

In this note we are interested in finding weak solutions to the Dirichlet problem for the quaternionic Monge-Amp\`ere operator in $\hn$ with a more degenerate right hand side and a continuous boundary data. It turns out to be possible whenever densities are in $L^p$ for $p>2$ and the exponent is optimal as we show.  To do that we follow the approach of Ko\l odziej from his papers \cite{K96, K98}. Probably the most interesting results are these which actually allow us to apply his method of proof. Among them is comparison of a quaternionic capacity and volume (Lebesgue measure). We prove it in the quaternionic setting coupling two things. Firstly the trick of Dinew and Ko\l odziej from \cite{DK14} which allowed them to show similar comparison for the capacity related to a complex Hesssian equation in $\cn$. It reduces to noting that although plurisubharmonic functions are rare among $m-$subharmonic ones still they realize this $m-$Hessian capacity. Secondly the fact that is interesting in its own right namely the comparison of complex and quaternionic Monge-Amp\`ere operators. To our knowledge it was not know or exploited before and rely on the observation that the Moore determinant of a hyperhermitian matrix is in fact the Pfaffian of an associated complex matrix. Afterwards we obtain an $L^\infty$ estimate for the solutions. The last step before proving the main theorem is stability of solutions in terms of their densities and boundary data but here the proofs are more standard. All of this is done in Section 4. In Section 3 we discuss the problem of finding the local integrability exponent for quaternionic plurisubharmonic functions. Proof of the main theorem there is inspired by the one presented in \cite{H07} for plurisubharmonic function in $\cn$. It turns out that the class of quaternionic plurisubharmonic functions exhibit an unusual property in this context namely the integrability exponent of a general function is two which is smaller than $2n$ occurring for a fundamental solution. This phenomenon can be excluded assuming boundedness of the function near the boundary of a domain what is proven in Section 4.

\bigskip

\textbf{Acknowledgments:} We wish to express our gratitude to S. Ko\l odziej for his guidance and many helpful suggestions. We are greatly indebted to S. Dinew for stimulating discussions which made the presentation of the paper much clearer. This research was partially supported by NCBiR project \textit{Kartezjusz} POWR.03.02.00-00-I001/16-00 and the National Science Center of Poland grant number 2017/27/B/ST1/01145.

\section{Preliminaries}
General references for quaternionic linear algebra and basic properties of quaternionic plurisubharmonic functions are \cite{A03a, A03b, AV06} while for quaternionic pluripotential theory \cite{WZ15, WK17, WW17}. Let us fix the notation, for an algebra of quaternions \begin{center} $\hh = \{ x_0+x_1\ii+x_2\jj+x_3\kk \: | \: x_0,x_1,x_2,x_3 \in \rr \}$ \end{center} where $\ii, \jj, \kk$ satisfy quaternionic relations we consider $\hn$ as a \textit{right} quaternionic module. With such a choice we denote by $I, J, K$ the complex structures induced by $\ii, \jj, \kk$ when treating $\hn$ as a flat hypercomplex manifold. We introduce two coordinate systems, \begin{center} $\hn \ni  (q_i)_{i=0}^{n-1} \longmapsto 
(z_j)_{j=0}^{2n-1} \in \cc^{2n}$ \end{center} in such a way that $q_i = z_{2i}+ \jj z_{2i+1}$, this is a holomorphic chart for the complex structure $I$ and \begin{center} $\hn \ni  (q_i)_{i=0}^{n-1} \longmapsto 
(x_j)_{j=0}^{4n-1} \in \rr^{4n}$ \end{center} in such a way that $q_i = x_{4i}+x_{4i+1}\ii +x_{4i+2}\jj +x_{4i+3} \kk$, this is just a real chart. It is easy to see that $z_{j}=x_{2j}+(-1)^j x_{2j+1} \ii$ for $j=0,...,2n-1$. As always $\partial$ and $\bpar$ are the canonical differential operators induced by the complex structure $I$ and $d=\partial + \bpar $, $d^c=\ii(\bpar-\partial)$. We also introduce the twisted differential \begin{center}$\partial_J := J^{-1} \circ \overline{\partial} \circ J$, \end{center} considered in \cite{AV10, V02}, which plays the role of $\bpar$ in the hypercomplex setting (eg. quaternionic Dolbeault or Salamon complex). For its properties we refer to the mentioned papers. Most notably we will only use the following \begin{center} $\jpar: \Lambda^{k,0}_I (\hn) \rightarrow \Lambda^{k+1,0}_I(\hn)$ since $J: \Lambda^{p,q}_I(\hn) \rightarrow \Lambda^{q,p}_I(\hn)$ \\ $\partial \jpar + \jpar \partial = 0$ \\ $\jpar ^2=0$. \end{center} Later on it may happen frequently that we skip the subscript $I$ and understand that $\Lambda^{k,0}(\hn)$ come from considering bedegrees with respect to $I$.
One can check that for a smooth function $u: \hn \rightarrow \rr$ the following formulas hold \begin{center} $\partial u = \sum\limits_{i=0}^{2n-1} (\partial_{z_i} u) dz_i$ \\ $\bpar u = \sum\limits_{i=0}^{2n-1} 
(\partial_{\overline{z_i}} u) d\overline{z_i}$ \\ $\jpar u = \sum\limits_{i=0}^{2n-1} (-1)^{i+1} (\partial_{\overline{z_{i+(-1)^i}}}u) dz_i$ \\ $\partial \bpar u = \sum\limits_{i,j} (\partial_{z_i} \partial_{\overline{z_j}} u) dz_i \wedge d \overline{z_j}$ \\ $\partial \jpar u = \sum\limits_{i,j} \left( (-1)^{j+1} \partial_{z_i} \partial_{\overline{z_{j+(-1)^j}}}u \right) dz_i \wedge dz_j= \sum\limits_{i<j} \left( (-1)^{j+1} \partial_{z_i} \partial_{\overline{z_{j+(-1)^j}}}u - (-1)^{i+1} \partial_{z_j} \partial_{\overline{z_{i+(-1)^i}}} \right) dz_i \wedge dz_j$. \end{center} 
Suppose that $f: \hn \longrightarrow \hh$ is a $\mathcal{C}^2$ function, we define the formal quaternionic derivatives: 

\begin{center} $\frac{\partial f}{ \partial \bar{q_{\alpha}}}= \frac{\partial f}{\partial x_{4\alpha}} + \ii \frac{\partial f}{\partial 
x_{4\alpha+1}} + \jj \frac{\partial f}{\partial x_{4\alpha+2}} + \kk \frac{\partial f}{\partial x_{4\alpha+3}}$ \\ and \\ $\frac{\partial f}{ \partial q_{\alpha}}= \overline{\frac{\partial \bar{f}}{ \partial \bar{q_{\alpha}}}}= \frac{\partial f}{\partial x_{4 \alpha}} - \frac{\partial f}{\partial 
x_{4 \alpha +1}} \ii -  \frac{\partial f}{\partial x_{4 \alpha +2}} \jj - \frac{\partial f}{\partial x_{4 \alpha +3}} \kk$.\end{center}  
Let us observe that for any $f: \hn \longrightarrow \hh$ of class $\mathcal{C}^2$  \begin{center} $\frac{\partial }{ \partial \bar{q_{\alpha}}}\frac{\partial }{ \partial q_{\beta}}=\frac{\partial }{ \partial q_{\beta}}\frac{\partial }{ \partial \bar{q_{\alpha}}}.$ \end{center} Furthermore for a \underline{real} valued $f$ one has 
\begin{center} $\frac{\partial }{ \partial \bar{q_{\alpha}}}\frac{\partial }{ \partial q_{\alpha}}f=\frac{\partial^2 f}{\partial x_{4 \alpha}^2} +  \frac{\partial^2 f}{\partial 
x_{4 \alpha +1}^2} +  \frac{\partial^2 f}{\partial x_{4 \alpha + 2}^2} +  \frac{\partial^2 f}{\partial x_{4 \alpha +3}^2}$ and  $\frac{\partial }{ \partial \bar{q_{\alpha}}} \left(\frac{\partial }{ \partial q_{\beta}}f \right)=\overline{\frac{\partial }{ \partial \bar{q_{\beta}}}\frac{\partial }{ \partial q_{\alpha}}f}$. \end{center} As a consequence the matrix \begin{center} $Hess(f,\hh)= \left( \frac{\partial^2 f}{\partial \bar{q_\alpha}\partial q_\beta} \right)_{\alpha, \beta \in \{1,...,n\}}$ \end{center} is a hyperhermitian matrix for any real valued $f$. The following relations are known to hold for a smooth real valued function $u$ 
\begin{center} 
$(dd^c u)^{2n}=2^{2n}( \ii \partial \bpar u)^{2n}=4^{2n} (2n)! \det(\frac{\partial^2 u}{\partial_{z_i} \partial_{\overline{z_{j}}}}) ( \frac{\ii}{2} dz_0 \wedge d \overline{z_0})\wedge ... \wedge (\frac{\ii}{2} dz_{2n-1} \wedge d 
\overline{z_{2n-1}})$ \\ $(\partial \jpar u)^n=\frac{n!}{4^n} \det (\frac{\partial^2 u}{\partial_{\overline{q_{l}}} \partial_{q_k}})(dz_0 \wedge d z_1 \wedge ... \wedge dz_{2n-2} \wedge d {z_{2n-1}})$ \end{center} where in the last expression $\det$ is the Moore determinant, cf. \cite{M22} for the original definition, of a hyperhermitian matrix. The last formula was computed in \cite{AV06} and, in a different setting, in \cite{WW17}. For further simplifications we introduce some canonical differential forms \begin{center} $\omega_{2n}= \sum\limits_{i=0}^{2n-1} \frac{\ii}{2} dz_{i} \wedge d \overline{ z_{i}}$, $\beta_n=\sum\limits_{i=0}^{n-1} dz_{2i} \wedge d z_{2i+1}$, $\Omega_n = \frac{\beta_n^n}{n!}= dz_0 \wedge d z_1 \wedge ... \wedge dz_{2n-2} \wedge d {z_{2n-1}}$.\end{center} Since we will extensively use facts from pluripotential theory reproved in the quaternionic setting by Wan, Wang and Kang in \cite{WZ15, WW17} it is desirable to compare differential operators $\partial, \jpar$ which we use with their formally defined operators $d_0, d_1$. Those were introduced by D. Wan and W. Wang in \cite{WW17} to which we refer for more details. They consider the following "coordinates"  \begin{center} $z^{j0}=x_{2j} + (-1)^{j+1} x_{2j+1} \ii=\overline{z_j}$ \\ $z^{j1}= (-1)^{j+1} x_{2 \left(j + (-1)^j \right)} + x_{2 \left(j + (-1)^j \right) +1} \ii = (-1)^{j+1} z_{j+(-1)^j} $ \end{center} for $j=0,...,2n-1$ and the associated formal derivatives  \begin{center} $\nabla_{j0}= \partial_{x_{2j}} + (-1)^{j} \partial{x_{2j+1}} \ii = 2 \partial_{\overline{z_j}}$ \\ $\nabla_{j1}= (-1)^{j+1} \partial{x_{2 \left(j + (-1)^j \right)}} - \partial{x_{2 \left(j + (-1)^j \right) +1}} \ii = (-1)^{j+1} 2 \partial_{z_{j+(-1)^{j}}}$. \end{center} Afterwards they fix a complex basis $\omega^0,...,\omega^{2n-1}$ of $\cc^{2n} \approx {\cc^{2n}}^*$ and an associated one $\omega^I=\omega_{i_1}\wedge...\wedge \omega_{i_k}$, for $I=(i_1,...,i_k)$ such that $i_1<...<i_k$ belong to $\{0,...,2n-1\}$, of a complex exterior product $\Lambda^{k}(\cc^{2n}) \approx \Lambda^{k}({\cc^{2n}}^*)$. Finally they define operators 
$$d_i: \Lambda^{k,0}(\hn) \approx C^\infty(\hn,\Lambda^{k}\cc^{2n}) \rightarrow C^\infty(\hn,\Lambda^{k+1}\cc^{2n}) \approx  \Lambda^{k+1,0}(\hn)$$ for $i=0,1$ in the following way. Suppose that $F=\sum\limits_{I}f_I \omega^I$, then 
\begin{center} $d_iF= \sum\limits_{I,k\in \{0,...,2n-1\}} (\nabla_{ki}f_I) \omega^k\wedge\omega^I$. \end{center} From formulas for $\nabla_{ki}$ we obtain 
\begin{center} $d_0F= \sum\limits_{I,k\in \{0,...,2n-1\}} (\nabla_{k0}f_I) \omega^k\wedge\omega^I = \sum\limits_{I,k\in \{0,...,2n-1\}} 2 \left( \partial_{\overline{z_k}}f_I \right) \omega^k\wedge\omega^I$ \\ $d_1F= \sum\limits_{I,k\in \{0,...,2n-1\}} (\nabla_{k1}f_I) \omega^k\wedge\omega^I = \sum\limits_{I,k\in \{0,...,2n-1\}} 2 (-1)^{k+1} \left( \partial_{z_{k+(-1)^k}}f_I \right) \omega^k\wedge\omega^I$. \end{center} 

\begin{proposition} \label{zgodnosc} For the basis $\omega^k=(-1)^{k} dz_{k+(-1)^k}$ \begin{center} $d_0=2 \jpar$, $d_1=-2 \partial$ and $\Delta  = d_0d_1 = 4 \partial \partial_J$. \end{center} \end{proposition}
\begin{proof} 

Let us recall that $\jpar=J^{-1} \circ \overline{\partial} \circ J$ and one can check that $J$ acts as \begin{center} $J (d z_{2i+1})=d \overline{z_{2i}}$, $J(d z_{2i})=-d\overline{z_{2i+1}}$ i.e. $J(dz_k)=(-1)^{k+1}d\overline{z_{k+(-1)^k}}$. \end{center} 
As before, for $F=\sum\limits_{I}f_I \omega^I$, we obtain 
\begin{center} $\partial F = \sum\limits_{I,k\in \{0,...,2n-1\}} (\partial_{z_k}f_I) dz_k \wedge\omega^I$ \\ $\:$ \\ $\jpar F = J^{-1} \circ \overline{\partial}(\sum\limits_{I}f_I J(\omega^I))=J^{-1} \left( \sum\limits_{I,k \in \{0,...,2n-1 \}} (\partial_{\overline{z_k}}f_I) d\overline{z_k} \wedge J(\omega^I)  \right) = J^{-1} \left( \sum\limits_{I,k \in \{0,...,2n-1 \}} (\partial_{\overline{z_{k+(-1)^k}}}f_I) d\overline{z_{k+(-1)^k}} \wedge J(\omega^I)  \right)= \sum\limits_{I,k \in \{0,...,2n-1 \}} (\partial_{\overline{z_{k+(-1)^k}}}f_I) J^{-1}(d\overline{z_{k+(-1)^k}}) \wedge \omega^I  = \sum\limits_{I,k \in \{0,...,2n-1 \}} (\partial_{\overline{z_{k+(-1)^k}}}f_I) (-1)^{k+1}dz_{k} \wedge \omega^I$. \end{center} 
This results in \begin{center} $d_0 F = 2 \sum\limits_{I,k\in \{0,...,2n-1\}} \left( \partial_{\overline{z_k}}f_I \right) \omega^k\wedge\omega^I = 2 \sum\limits_{I,k\in \{0,...,2n-1\}} (-1)^{k} \left( \partial_{\overline{z_k}}f_I \right) dz_{k+(-1)^k} \wedge\omega^I = 2 \sum\limits_{I,k\in \{0,...,2n-1\}} (-1)^{k+1} \left( \partial_{\overline{z_{k+(-1)^k}}}f_I \right) dz_{k} \wedge\omega^I = \textbf{2} \jpar F$ \\ $d_1 F = 2 \sum\limits_{I,k\in \{0,...,2n-1\}} (-1)^{k+1} \left( \partial_{z_{k+(-1)^k}}f_I \right) \omega^k\wedge\omega^I  = 2 \sum\limits_{I,k\in \{0,...,2n-1\}} (-1) \left( \partial_{z_{k+(-1)^k}}f_I \right) dz_{k+(-1)^k} \wedge\omega^I =  \textbf{-2} \partial F$. \end{center}  
\end{proof}
\begin{remark}
Let us just emphasize that the choosing of $\partial, \jpar$ over $d_0, d_1$ has some deeper than just conventional meaning. These are the natural intrinsic operators not only in $\hn$ but on any hypercomplex manifold. In fact on an abstract hypercomplex manifold quaternionic plurisubharmonic functions are defined only with their aid, cf. \cite{AV06}, since the local chart definition is not possible due to non-integrability of a generic hypercomplex structure i.e. non-existence of quaternionic charts.
\end{remark}

From Proposition \ref{zgodnosc} it follows that we are able to use all results from \cite{WZ15, WK17, WW17} as well as from \cite{A03a, A03b, AV06}. We just give here the necessary details and refer to the mentioned papers for more of them. The quaternionic plurisubharmonic functions were defined by Alesker in \cite{A03a}. 
\begin{definition*}
Let $\Omega$ be a domain in $\hn$. We call an upper semi-continuous function $f: \Omega \rightarrow \rr$ (strictly) quaternionic plurisubharmonic, $qpsh$ for short, if $f$ restricted to any affine right quaternionic line intersected with $\Omega$ is (strictly) subharmonic as a function on a domain in $\rr^4$. The set of all $qpsh$ functions on $\Omega$ is denoted by $\qpsh(\Omega)$.
\end{definition*} 
\begin{remark} If we fix $t \in \{a \ii + b \jj + c \kk \: | \: a^2 +b^2 +c^2 = 1 \}$ an imaginary unit and consider $\hn$ as a complex vector space where multiplication by $\ii$ is given by a right multiplication by $t$ then $psh$ functions with respect to this complex structure are $qpsh$ since quaternionic lines are complex two planes. We will use that remark only for $t=\ii$ i.e. only for $\hn$ treated as $\cc^{2n}$ via the chart introduced in the begging of the paragraph.\end{remark}
For a smooth function being $qpsh$ is equivalent to $\partial \jpar u \geq 0$ in a quaternionic sense. Let us elaborate on it. The cones of strongly positive $SP^{2k} (\Omega) \subset \Lambda^{2k,0}_\rr(\Omega)$ and positive $\Lambda^{2k,0}_{\rr, \geq 0}(\Omega) \subset \Lambda^{2k,0}_\rr(\Omega)$ forms were introduced in \cite{AV06}, see also \cite{V10} for a careful and extended treatment. Here $\Lambda^{2k,0}_\rr(\Omega) \subset \Lambda^{2k,0}(\Omega)$ is the space of forms $\alpha$ such that $\overline{J(\alpha)}=\alpha$. To introduce them we firstly argue for a point, an element $\Omega_n \in \Lambda^{2n,0}_\rr(\hn \approx T_0 \hn)$ is chosen to be strongly positive and a convex combination of elements of the form $G^*(\Omega_k)$ for $G: \hn \rightarrow \hh^k$ a quaternionic linear map is strongly positive. When the reasoning is applied pointwise we obtain the notion of strong positivity for differential forms in $\Omega$. As always the cone of positive elements is the dual one. We have mentioned above that $(\partial \jpar u)^n$ agree with Moore's determinant of a quaternionic Hessian $Hess(u, \hh)$ for a smooth function, in \cite{A03b} Alesker motivated by \cite{BT76} showed that $(\partial \jpar u)^n$ can be interpreted as a measure for continuous $u$ and proved certain convergence for this operator. It is a cornerstone for having proper pluripotential theory. Later in \cite{WW17} authors proved that $\partial \jpar u$ is a positive current (where positivity is defined using the cone of strongly positive forms) for any $qpsh$ function. What is more important they showed that like in the complex case, cf. \cite{BT82}, one can define $(\partial \jpar u)^n$ for any locally bounded $u$ and treat it as a measure. From there one can recreate most of theorems which hold for $psh$ functions. Among other things they have shown weak convergence of this operator on decreasing sequences of $qpsh$ functions and Chern-Levine-Nirenberg inequalities, cf. \cite{WW17}. In \cite{WZ15} the quaternionic relative capacity is introduced in the spirit of Bedford and Taylor, let $K \subset \Omega$ be a compact set then \begin{center} $cap(K,\Omega)=sup \left\{ \int_K (\partial \jpar u)^n \: | \: u \in \qpsh(\Omega), \: 0 \leq u \leq 1 \right\} $ \end{center} and this can be extended to Borel subsets as well. What is more authors prove quasicontinuity of $qpsh$ functions and most notably the comparison principle which is probably the most powerful tool in pluripotential theory. The statement is exactly as we know it in the complex case but we recall it for reader's convenience. 
\begin{theorem*}{\cite{WZ15}} Let $u,v \in \qpsh(\Omega) \cap L^\infty_{loc}(\Omega)$. If for any $\xi \in \partial \Omega$, \begin{center} $\liminf\limits_{\xi \leftarrow q \in \partial \Omega} (u(q)-v(q)) \geq 0$ \end{center} 
then \begin{center} $\int\limits_{\{u<v\}} (\partial \jpar v)^n \leq \int\limits_{\{u<v\}} (\partial \jpar u)^n $. \end{center} In particular if $ (\partial \jpar v)^n \geq (\partial \jpar u)^n$ as measures then $u \geq v$ in 
$\Omega$.
\end{theorem*} Finally they characterize maximality of a bounded $qpsh$ function in terms of vanishing of its Monge-Amp\`ere mass. Here we mean that $u \in \qpsh(\Omega)$ is maximal if it is above any other $qpsh$ function on compacts $K \subset \Omega$ provided the values of both functions are the same on $\partial K$.

\section{Local integrability of $qpsh$ functions}
In this section we address the question of local integrability of  $qpsh$ functions in a domain $\Omega \subset \hn$. For $psh$ functions it is well know that they are locally integrable with any exponent. 
The proof of the proposition below is inspired by the presentation in \cite{H07}.

\begin{proposition} \label{hor} Suppose $u \in \qpsh(\Omega)$ is such that $u \not \equiv - \infty$. Then $u \in L^p_{loc}(\Omega)$ for any $p<2$ and the bound on $p$ is optimal. What is more if $u_j \not \equiv -\infty$ is a sequence of $qpsh$ functions converging in $L^1_{loc}(\Omega)$ to some $u$, necessarily belonging to $\qpsh(\Omega)$, then convergence holds in $L^p_{loc}(\Omega)$ for any $p<2$. \end{proposition}  
\begin{proof}
Suppose w.l.o.g. that $u \leq 0$ in a neighborhood of a quaternionic polyball $P(0,1)$ of radius one centered at $0$ contained in $\Omega$, that $u(0)>-\infty$ and fix $p<2$. Let us deal firstly with the case $n=1$. From the Riesz representation theorem, cf. Theorem 3.3.6 in \cite{H07}, $$u(q)=h(g) + 
\int\limits_{ \parallel \xi \parallel <1} G(q,\xi) d \mu(\xi)$$ for some non positive harmonic function $h$ in $B(0,1):=B_1$, non negative Borel measure $\mu$ and Green's function $G(q,\xi) = -\frac{1}{\parallel q- \xi \parallel^2} + \frac{1}{\parallel (q- \frac \xi {|\xi|^2})|\xi| \parallel^2}$. 

By Harnack's inequality, cf. Theorem 3.1.7 in \cite{H07}, for any $\parallel q \parallel < \frac{1}{2}$ we have $$0 \leq -h(q) \leq 
\frac{1+\parallel q \parallel}{(1-\parallel q \parallel)^3}(-h(0)) \leq 12 (-h(0)).$$ This shows that 
$$\parallel h \parallel_{L^p \left(B(0,\frac{1}{2}) \right)} \leq C_p |h(0)|$$ 
for a constant $C_p$, depending only on $p<2$, which we may still need to increase (see below). 

For estimating the second component of the decomposition of $u$ let us introduce the following notation \begin{center} $H(q,\xi)=-G(q, \xi) \geq 0$, for $\xi=0$ we have $ H(q,0)=\frac 1 {\parallel q \parallel^2} - 1$.\end{center} We consider two cases depending on whether $\xi$ is close to the center or to the boundary of $B_1$. 

In the first case, say when $\parallel \xi \parallel \leq \frac 3 4$, we use the estimate \begin{center} $0 \leq H(q,\xi) \leq  \frac{1}{\parallel q - \xi \parallel^{2}}$ \end{center} for any $q$ and $ \xi$, consequently \begin{center} $\left( \: \int\limits_{\parallel q \parallel < \frac{1}{2}} \left( H(q,\xi) \right)^p d \mathcal{L}^4(q) \right)^\frac 1 p \leq \left( \: \int\limits_{\parallel q \parallel < \frac{1}{2}} \frac{1}{\parallel q - \xi \parallel^{2p}} d \mathcal{L}^4(q) \right)^\frac 1 p \leq \left( \: \int\limits_{\parallel q \parallel < \frac{5}{4}} \frac{1}{\parallel q \parallel^{2p}} d \mathcal{L}^4(q) \right)^\frac 1 p \leq C'_p \left( \frac{1}{\parallel \xi \parallel^2}-1 \right)$ \end{center} for a constant $C'_p$ independent of $\xi$ and depending only on $p<2$, since the expression $\frac{1}{\parallel \xi \parallel^2}-1$ is bounded from below for $\parallel \xi \parallel \leq \frac 3 4$.

In the second case, say when $\parallel \xi \parallel \geq \frac 3 4$, we note that for any fixed $\xi$ the function $H(\cdot,\xi)$ is non negative and harmonic in $B_{\frac 3 4}$. Applying Harnack's inequality for each fixed $\xi$ we obtain that for all $\parallel \xi \parallel \geq \frac 3 4$ and for all $\parallel q \parallel < \frac 3 4$  \begin{center} $ 0 \leq H(q,\xi) \leq \left( \frac 3 4 \right)^2 \frac{\frac 3 4 + \parallel q \parallel }{(\frac 3 4 - \parallel q \parallel )^3} H(0,\xi)$, \end{center} hence for all $\parallel \xi \parallel \geq \frac 3 4$ and $\parallel q \parallel < \frac 1 2$ \begin{center} $ 0 \leq H(q,\xi) \leq 45 \left( \frac{1}{\parallel \xi \parallel^2}-1 \right) $.\end{center}

To sum up we have proven that there exists a constant $C_p=\max \{C'_p, 45\}$, independent of $\xi$, such that for $\parallel \xi \parallel < 1$ \begin{center} $\parallel H(\cdot,\xi) \parallel_{L^p \left(B(0,\frac{1}{2}) \right)} \leq C_p \left( \frac{1}{\parallel \xi \parallel^2}-1 \right)$. \end{center}

From Minkowski's inequality and Minkowski's integral inequality we obtain
\begin{center} $\parallel u \parallel_{L^p \left( B(0,\frac{1}{2}) \right)} \leq \parallel h \parallel_{L^p \left(B(0,\frac{1}{2}) \right)} + \left( \int\limits_{\parallel q \parallel < \frac{1}{2}} \mid \int\limits_{ \parallel \xi \parallel <1} H(q,\xi) d \mu(\xi) \mid^p  d \mathcal{L}^4(q) \right)^\frac 1 p \leq C_p |h(0)| + \int\limits_{ \parallel \xi \parallel <1} \left( \int\limits_{\parallel q \parallel < \frac{1}{2}} H(q,\xi)^p d \mathcal{L}^4(q) \right)^\frac 1 p  d \mu(\xi)  \leq C_p \left( |h(0)| + \int\limits_{\parallel \xi \parallel < 1} \left(\frac{1}{\parallel \xi \parallel^2}-1 \right) d \mu(\xi) \right)=C_p |u(0)|$. \end{center} Using Fubini's theorem and the estimate above one obtains that in the case of $n \geq 1$ we have $$\parallel u \parallel_{L^p \left( P(0,\frac{1}{2}) \right)} \leq C_p^n |u(0)|.$$ 

To the end observe that $\Omega'$, the set of points in $\Omega$ in neighborhood of which $u$ is integrable with exponent $p$, is an open set by definition. It is closed by what we have just shown. This is so because if $q \in \overline{\Omega'}$ and $r>0$ is such that $P(q,3r) \subset \subset \Omega$ then we can find and element $q'$ of $\Omega'$ within $\frac r 2$ distance from $q$ and a point $q'' \in \Omega$ within $\frac r 2$ distance from $q'$ such that $u(q'')$ is finite. We note that $P(q'',2r) \subset \subset \Omega$, consequently $u$ is integrable with the exponent $p$ on $P(q'',r)$, and $q \in P(q'',r)$. What is more $\Omega'$ is nonempty by the assumption $u \not \equiv -\infty$. The bound on $p$ is optimal as the example of $-\frac{1}{ \parallel q_0 \parallel ^2}$ in $\hn$ for $n\geq 1$ shows.   

For the proof of the second assertion we note that the sequence $u_j-u$ is bounded in $L^p_{loc}(\Omega)$ for any $1 \leq p < 2$. To prove this it is enough to show that the sequence $u_j$ is bounded in $L^p_{loc}(\Omega)$. Fix any point $q \in \Omega$ and let $r>0$ be such that $P(q,3r) \subset \subset \Omega$. We claim that $L^p$ norms of $u_j$ in $P(q,\frac r 2)$ are bounded. Suppose to the contrary that they are not. Let us choose a subsequence $j_k$ such that $\parallel u_{j_k} \parallel_{L^p(P(q,\frac r 2))} \longrightarrow \infty$. We know that $u_{j_k}$'s are locally uniformly bounded from above, cf. Theorem 3.2.13 in \cite{H07}, so we may assume that $u$ and $u_{j_k}$'s are non positive in $P(q,3r)$. It is possible to find a point $q'$ within $\frac r 2$ distance from $q$ such that $u(q') > - \infty$ and $\limsup\limits_{k \rightarrow \infty}u_{j_k}(q')=u(q')$ since both this properties hold almost everywhere in $\Omega$, cf. Theorem 3.2.13 in \cite{H07}. We assume w.l.o.g that $\lim\limits_{k \rightarrow \infty}u_{j_k}(q')=u(q')$ for if not we take a subsequence again. In particular there exists $C>0$ such that $u_{j_k}(q')>-C$ for any $k$. This together with the estimate we have proven shows that the sequence $\parallel u_{j_k} \parallel_{L^p(P(q',r))}$ is bounded. Because we also know that $P(q,\frac r 2) \subset P(q',r)$ contradiction with $\parallel u_{j_k} \parallel_{L^p(P(q,\frac r 2))} \rightarrow \infty$ is obtained. Fix $1 \leq p <2$ and observe that for any compact $K \subset \Omega$ we have     

\begin{center} $\int_K |u_j-u|^p d \mathcal{L}^{4n} = \int_K |u_j-u|^{\frac{2-p}{2}}|u_j-u|^{\frac{3p-2}{2}} d 
\mathcal{L}^{4n} \leq \left( \int_K |u_j-u|^{(\frac{2-p}{2})(\frac{2}{2-p})} d \mathcal{L}^{4n} \right)^{(\frac{2-p}{2})} \left( \int_K |u_j-u|^{(\frac{3p-2}{2}) (\frac{2}{p})} d \mathcal{L}^{4n} \right)^{\frac{p}{2}}= 
\left( \int_K |u_j-u| d \mathcal{L}^{4n} \right)^{(\frac{2-p}{2})} \left( \int_K |u_j-u|^{(3-\frac{2}{p})} d \mathcal{L}^{4n} \right)^{\frac{p}{2}}$ 
\end{center} 
by H\"older's inequality. By the assumption the first term tends to zero while second one is bounded since $1 \leq 3-\frac{2}{p}<2$. This proves that $u_j$ tend to $u$ in $L^p_{loc}(\Omega)$ for any $1 \leq p < 2$. 
\end{proof}
The following proposition was proven in \cite{WW17}.
\begin{proposition}{\cite{WW17}}
The function $f(q)=-\frac{1}{\parallel q \parallel^2 }$ is a fundamental solution for the quaternionic Monge-Amp\`ere operator in $\hn$. More exactly $$(\partial \jpar f)^n= \frac{2^n \pi^{2n}n!}{(2n)!}\delta_0.$$
\end{proposition}
We see that the fundamental solution to the quaternionic Monge-Amp\`ere equation is in $L^p_{loc}(\hn)$ for any $p<2n$ while a generic $qpsh$ function only for $p<2$ which is in contrast with the case of $psh$ functions. 

\section{Dirichlet problem for quaternionic Monge-Amp\`ere equation}

In this section we aim to solve the Dirichlet problem  \begin{center} $\begin{cases} u \in \qpsh(\Omega) \cap C(\overline{\Omega}) \\ (\partial \jpar u)^n = f \Omega_{n} \\ u_{|\partial \Omega}=\phi \in C(\partial \Omega) \end{cases}$ \end{center} where $f \in L^q(\Omega)$ for $q>2$ and $\Omega \subset \subset \hn$ is a smoothly bounded, strictly quaternionic pseudoconvex domain, which is a global assumption for $\Omega$ in this section. Let us recall that

\begin{definition*} $\Omega \subset \subset \hn$ a smoothly bounded domain is strictly quaternionic pseudoconvex if there exists $v$, a smooth strictly $qpsh$ function in a neighborhood of $\overline{\Omega}$, such that $v<0$ in $\Omega$, $v=0$ but $\nabla v \not = 0 $ on $\partial \Omega$. \end{definition*} Let us just mention that the Dirichlet problem for the complex Monge-Amp\`ere equation with densities in $L^p$ for $p>1$ was solved by Ko\l odziej in \cite{K96}. In fact he proved it for densities in appropriate Orlicz spaces being subspaces of $L^1$ and in particular cases reducing to $L^p$. For the real Monge-Amp\`ere equation one can always solve the above problem for any density in $L^1$, cf. \cite{RT77}. 

The first goal is to compare complex and quaternionic Monge-Amp\`ere operators. We start with smooth functions in which case we have to compare complex and quaternionic Hessians or rather their determinants to be precise. 
 
\begin{lemma} \label{q-c} For a smooth function $u: \Omega \rightarrow \rr$ and any $l,k \in \{0,...,n-1 \}$ 
\begin{center}
$\partial_{\overline{q_{l}}} \partial_{q_k} u=\partial_{\overline{q_{l}}} \left( \partial_{x_{4k}} u- \ii \partial_{x_{4k+1}} u - \jj \partial_{x_{4k+2}}u - \kk \partial_{x_{4k+3}}u \right)= \partial_{\overline{q_{k}}} \left( 2 \partial_{z_{2k}} u - 2\jj \partial_{\overline{z_{2k+1}}} u  \right)= \left( 2 \partial_{\overline{z_{2l}}} + 2 \jj \partial_{z_{2l+1}} \right) \left( 2 \partial_{z_{2k}} u - 2\jj \partial_{\overline{z_{2k+1}}} u  \right)= 4 \left( \partial_{\overline{z_{2l}}} \partial_{z_{2k}} u + \partial_{z_{2l+1}} \partial_{\overline{z_{2k+1}}} u \right) +4\jj \left( \partial_{\overline{z_{2l+1}}} \partial_{z_{2k}} u -\partial_{z_{2l}} \partial_{\overline{z_{2k+1}}} u \right)$ \end{center}
\end{lemma} 
Let us recall that we distinguish the set $\psh(\Omega)$ of plurisubharmonic functions in $\Omega$ by identifying $\hn$ with $\cc^{2n}$ via a chart introduced in Section 2.
\begin{lemma} \label{cpr} For a function $u \in \psh(\Omega) \cap C^2(\Omega) \subset \qpsh(\Omega)$ the following holds \begin{center} $\left( \det (\frac{\partial^2 u}{\partial_{\overline{q_{l}}} \partial_{q_k}}) \right)^2 \geq 4^{2n} \det(\frac{\partial^2 u}{\partial_{z_i} \partial_{\overline{z_{j}}}})$. \end{center} \end{lemma}
\begin{proof}
\noindent Let us denote \begin{center} $Hess(u,\cc)=\left( \frac{\partial^2 u}{\partial_{z_i} \partial_{\overline{z_{j}}}} \right)_{i,j=0,...,2n-1}$ and $Hess(u,\hh)= \left(\frac{\partial^2 u}{\partial_{\overline{q_{l}}} \partial_{q_k}} \right)_{l,k=0,...,n-1}$.\end{center} Note that $$\det \left( \frac{\partial^2 u}{\partial_{z_i} \partial_{\overline{z_{j}}}} \right)_{i,j=0,...,2n-1} = \det Hess(u,\cc)= \det \overline{Hess(u,\cc)}= \det \left( \frac{\partial^2 u}{\partial_{\overline{z_i}} \partial_{z_{j}}} \right)_{i,j=0,...,2n-1}.$$ The last matrix is Hermitian positive since it is just $Hess(u,\cc)^T$. If $Hess(u,\hh)=G+\jj H$ then we define \begin{center} $\psi \left( Hess(u,\hh) \right)= \begin{pmatrix} G & - \overline{H} \\ H & \overline{G} \end{pmatrix}$. \end{center} By Lemma \ref{q-c} we obtain that 
\begin{center} $\psi \left( Hess(u,\hh) \right)=
4 \begin{pmatrix} \left[ \partial_{\overline{z_{2l}}} \partial_{z_{2k}} u + \partial_{z_{2l+1}} \partial_{\overline{z_{2k+1}}} u \right]_{l,k} & \left[ - \partial_{{z_{2l+1}}} \partial_{\overline{z_{2k}}} u + \partial_{\overline{z_{2l}}} \partial_{{z_{2k+1}}} u \right]_{l,k} \\ \left[ \partial_{\overline{z_{2l+1}}} \partial_{z_{2k}} u -\partial_{z_{2l}} \partial_{\overline{z_{2k+1}}} u \right]_{l,k} & \left[ \partial_{{z_{2l}}} \partial_{\overline{z_{2k}}} u + \partial_{\overline{z_{2l+1}}} \partial_{{z_{2k+1}}} u \right]_{l,k} \end{pmatrix}=
4\begin{pmatrix} \left[ \partial_{\overline{z_{2l}}} \partial_{z_{2k}} u \right]_{l,k} & \left[\partial_{\overline{z_{2l}}} \partial_{{z_{2k+1}}} u \right]_{l,k} \\ \left[ \partial_{\overline{z_{2l+1}}} \partial_{z_{2k}} u \right]_{l,k} & \left[ \partial_{\overline{z_{2l+1}}} \partial_{{z_{2k+1}}} u \right]_{l,k} \end{pmatrix} 
+ 4 \begin{pmatrix} \left[ \partial_{z_{2l+1}} \partial_{\overline{z_{2k+1}}} u \right]_{l,k} & \left[ - \partial_{{z_{2l+1}}} \partial_{\overline{z_{2k}}} u \right]_{l,k} \\ \left[ -\partial_{z_{2l}} \partial_{\overline{z_{2k+1}}} u \right]_{l,k} & \left[ \partial_{{z_{2l}}} \partial_{\overline{z_{2k}}} u \right]_{l,k} \end{pmatrix}$. \end{center} 
Following \cite{CP92} we introduce three matrices \begin{center}$A=\left[ \partial_{\overline{z_{2l}}} \partial_{z_{2k}} u \right]_{l,k}$, $B= \left[ \partial_{\overline{z_{2l+1}}} \partial_{z_{2k}} u \right]_{l,k}$, $C=\left[ 
\partial_{\overline{z_{2l+1}}} \partial_{{z_{2k+1}}} u \right]_{l,k}$. \end{center} Under this notation \begin{center} $\psi \left( Hess(u,\hh) \right)= 4 \begin{pmatrix} A & \overline{B}^T \\ B & C \end{pmatrix} + 4  
\begin{pmatrix} \overline{C} & - \overline{B} \\ -B^T & \overline{A} \end{pmatrix} $. \end{center} Note that $\begin{pmatrix} A & \overline{B}^T \\ B & C \end{pmatrix}$ is the conjugate of a Hessian of $u$ with respect to the 
coordinates $z_0$, ..., $z_{2n-2}$, $z_1$, ..., $z_{2n-1}$, so it is Hermitian positive as well. Moreover \begin{center} $\det \left( \frac{\partial^2 u}{\partial_{z_i} \partial_{\overline{z_{j}}}} \right) = \det \begin{pmatrix} A & 
\overline{B}^T \\ B & C \end{pmatrix}$.\end{center} Consider the matrix $I=\begin{pmatrix}0 & -I_n \\ I_n & 0 \end{pmatrix}$ with the inverse $I^{-1}=\begin{pmatrix}0 & I_n \\ -I_n & 0 \end{pmatrix}$ and the determinant equal to one. Note that \begin{center}
$I  \begin{pmatrix} \overline{C} & - \overline{B} \\ -B^T & \overline{A} \end{pmatrix} I^{-1}=  \begin{pmatrix} \overline{A} & {B}^T \\ \overline{B} & \overline{C} \end{pmatrix}$, \end{center} and the last matrix is the conjugate of the one just shown to be Hermitian positive so as such is also Hermitian positive. Consequently $\begin{pmatrix} \overline{C} & - \overline{B} \\ -B^T & \overline{A} \end{pmatrix}$ is positive as being similar to the one of that kind. Now we use the equality between Moore's determinant of a matrix $M$ and the Pfaffian of an associated complex matrix $I \psi(M)$ as proved in \cite{D70} which results in  \begin{center} $\left( \det \left(\frac{\partial^2 u}{\partial_{\overline{q_{l}}} \partial_{q_k}} \right) \right)^2 = \det \psi \left( Hess(u,\hh) \right) = 4^{2n} \det \left( \begin{pmatrix} A & \overline{B}^T \\ B & C \end{pmatrix} + 
\begin{pmatrix} \overline{C} & - \overline{B} \\ -B^T & \overline{A} \end{pmatrix} \right) \geq 4^{2n} \det \begin{pmatrix} A & \overline{B}^T \\ B & C \end{pmatrix} = 4^{2n} \det(\frac{\partial^2 u}{\partial_{z_i} 
\partial_{\overline{z_{j}}}})$ \end{center} as we desired to prove.
\end{proof}
Having this the announced comparison of quaternionic and complex Monge-Amp\`ere operators for non smooth functions follows from the standard approximation procedure as presented in the proof below. Real and quaternionic Monge-Amp\`ere operators were compared by Wan in \cite{W18}.
\begin{theorem} \label{cpo}
Let $u \in \psh(\Omega) \cap C(\overline{\Omega})$ 
satisfy the equation $$(dd^cu)^{2n}= f^2 4^{2n} \omega_{2n}^{2n}$$ 
for some non negative $f \in L^p(\Omega)$, $p > 2$. 
Then $$(\partial \partial_J u)^n \geq f \Omega_n^n .$$ 
\end{theorem}
\begin{proof}
Since the property is local we may assume that $\Omega$ is strictly pseudoconvex, otherwise we argue as below but for some ball contained in $\Omega$. Approximate $f$ by a sequence of smooth positive functions $f_i$ in $L^p$ norm and $u$ uniformly by a sequence of smooth functions $\phi_i$ on $\partial \Omega$. Let us solve the family of Dirichlet problems \begin{center} $\begin{cases} u_i \in \psh(\Omega) \cap C^\infty(\overline{\Omega}) \\ (dd^c u_i)^{2n}=f_i^2 4^{2n}\omega_{2n}^{2n} \\ u_i = \phi_i \: on \: \partial \Omega \end{cases}$\end{center} which is possible due to \cite{CKNS85}. Observe that $u_i$ converge uniformly to $u$ due to stability of solutions in $L^q$, $q>1$ for the complex Monge-Amp\`ere equation, cf. \cite{K96, DK14}. From Lemma \ref{cpr} \begin{center} $(\partial \jpar u_i)^n=\frac{1}{4^n} \det (\frac{\partial^2 u_i}{\partial_{\overline{q_{l}}} \partial_{q_k}}) \Omega_n^n \geq \sqrt{\det(\frac{\partial^2 u_i}{\partial_{z_m} \partial_{\overline{z_{n}}}})}=f_i \Omega_n^n$ \end{center} as measures. Right hand sides converge as measures to $f\Omega_n^n$  and left ones converge to $(\partial \partial_J u)^n$ since convergence of $u_i$ is uniform, cf. \cite{WW17}, what ends the proof.
\end{proof}
We are going to prove an inequality between the volume and quaternionic capacity which was an essential component of Ko\l odziej's proof of solvability of the complex Monge-Amp\`ere equation for densities in appropriate Orlicz spaces, cf. \cite{K96, K05}. 
Similar inequality for the capacity associated to a complex $m-$Hessian equation was proven in \cite{DK14} with the usage of an observation that $psh$ functions although being an extremal example of $m-$subharmonic ones still realize the $m-$Hessian capacity. Here we couple that trick with comparison of quaternionic and complex Monge-Amp\`ere operators proved in Theorem \ref{cpo}. 
\begin{lemma} \label{c-v}
For a fixed $p \in (1,2)$ there exists a constant $C(p,R)$ such that for any $\Omega \subset B(0,R)$ and $K \subset \subset \Omega$ \begin{center} $ \leb^{4n}(K) \leq  C(p,R) cap^p(K,\Omega)$. \end{center}
\end{lemma}
\begin{proof}
Suppose that $\leb(K) \not = 0$ otherwise there is nothing to prove. Take any $\epsilon \in (0,\frac{1}{2})$ and consider $f=\leb(K)^{2 \epsilon - 1} \chi_K$. Let us solve the Dirichlet problem \begin{center} $\begin{cases} u \in \psh(B) \cap C(\overline{B}) \\ (dd^c u)^{2n}=f4^{2n}\omega_{2n}^{2n} \\ u=0 \: on \: \partial B\end{cases}$ \end{center} which is possible due to \cite{C84}. By Theorem \ref{cpo} the quaternionic Monge-Amp\`ere operator of the solution $u$ satisfies $$(\partial \partial_J u)^n \geq \sqrt{f} \Omega_n^n.$$ Take $q=1+\epsilon$, one checks that \begin{center} $\int_B f^q \left( 4^{2n}(2n)! \right)^q d \leb^{4n}= \left( 4^{2n}(2n)! \right)^q \leb(K)^{(2 \epsilon - 1)(1+\epsilon)+1}= \left( 4^{2n}(2n)! \right)^q \leb(K)^{2 \epsilon^2 +\epsilon} \leq \left( 4^{2n}(2n)! \right)^2R^{4n}$ \end{center} i.e. the $L^q$ norm of $f$ is bounded by a quantity depending only on $R$. By Ko\l odziej's $L^\infty$ estimate, cf. \cite{K96, K98}, there exists a constant $c(\epsilon,R)$ such that $$\parallel u \parallel_{L^\infty(B)} \leq \frac{1}{c(\epsilon , R)}.$$ Put $v=c(\epsilon, R) u$, then since $v$ is a $qpsh$ function such that $-1 \leq v \leq 0$ $$cap(K,\Omega) \geq \int_K (\partial \jpar v)^n \geq n!c(\epsilon, R)^n \left(\leb^{4n}(K)\right)^{\frac{2\epsilon +1}{2}}$$ and consequently $$\left( \frac{1}{n!c(\epsilon, R)^n}\right)^{\frac{2}{2\epsilon+1}} cap^{\frac{2}{2\epsilon+1}}(K,\Omega) \geq \leb^{4n}(K).$$ This gives the claim since when $\epsilon$ vary in $(0,\frac 1 2 )$ the exponent $\frac{2}{2 \epsilon+1}$ vary in $(1,2)$.
\end{proof}
In the previous section we have proven that any $qpsh$ function belongs to $L^p$ for $p<2$ locally and that this is the optimal exponent. The lemma below gives the estimates on capacity and volume for sublevel sets of certain $qpsh$ functions. In particular it shows that in the case of  $u \in \qpsh(\Omega)$ bounded near the boundary of $\Omega$ the local integrability of $|u|^p$ is ensured for $p<2n$. Again this bound is optimal as the example of $-\frac{1}{\parallel q \parallel^2}$ shows.   
\begin{lemma} \label{bded}
Fix $p \in (1,2)$. Let $u \in \qpsh(\Omega) \cap L^\infty_{loc}(\Omega)$ be such that \begin{center} $\liminf\limits_{q \rightarrow q_0} \left( u(q) - v(q) \right) \geq 0$ \end{center} for any $q_0 \in \partial \Omega$ and some fixed $v \in \qpsh(\Omega) \cap C(\overline{\Omega})$. Then there exists a constant $C(p,diam (\Omega))$ depending only on $p$ and the diameter of $\Omega$ such that for $U(s)=\{u<v-s\} \subset \subset \Omega$ \begin{center} $cap(U(s),\Omega) \leq \frac{\int_\Omega (\partial \jpar u)^n}{s^n}$ and $\leb^{4n}(U(s)) \leq C(p, diam (\Omega))\frac{\int_\Omega (\partial \jpar u)^n}{s^{pn}}$.  \end{center}
\end{lemma}
\begin{proof}
Take $\epsilon >0$ and a compact $K \subset U(s)$. By definition one can find $w \in \qpsh(\Omega) \cap L^\infty_{loc}(\Omega)$ such that $-1 \leq w \leq 0$ and $$\int_K (\partial \jpar w)^n \geq cap(K,\Omega) - \epsilon .$$
Due to the way we have chosen $K$ and the comparison principle  \begin{center} $cap(K,\Omega) - \epsilon \leq \int_K (\partial \jpar w)^n \leq \int_{\{\frac{u}{s}<\frac{v}{s}-1\}} (\partial \jpar w)^n \leq \int_{\{\frac{u}{s}<\frac{v}{s}+w\}} (\partial \jpar w)^n \leq \int_{\{\frac{u}{s}<\frac{v}{s}+w\}} \left( \partial \jpar (\frac{v}{s} + w) \right)^n \leq \frac{1}{s^n} \int_{\{\frac{u}{s}<\frac{v}{s}+w\}} (\partial \jpar u)^n \leq 
\frac{\int_\Omega (\partial \jpar u)^n}{s^n}$. \end{center} Letting $\epsilon$ tend to $0$ and taking the supremum over all compacts $K$ we obtain the first claim. The second one follows from Lemma \ref{c-v}.
\end{proof}
The next goal is to prove the a priori $L^\infty$ estimate for continuous solutions of the Dirichlet problem.  Firstly note that by Alesker's result on the Dirichlet problem with continuous density and boundary value, cf. \cite{A03b}, and characterization of maximality of $qpsh$ functions as in \cite{WZ15} we can find $v \in C(\overline{\Omega})$ solving \begin{center} $\begin{cases} (\partial \jpar v)^n =0 \\ v_{|\partial \Omega}=\phi \in C(\partial \Omega) \end{cases}$ \end{center} i.e. being the maximal $qpsh$ function matching our boundary condition. For such a fixed $v$ we denote \begin{center} $U(s)=\{u<v-s\} \subset \Omega$ \end{center} and introduce the function \begin{center} $b(s)= \left( cap(U(s),\Omega) \right)^\frac{1}{n}$. \end{center}
\begin{theorem} \label{inf} There exists a constant $C \left(q,\parallel f \parallel_{L^q(\Omega)},\parallel \phi \parallel_{L^\infty(\partial \Omega)}, diam (\Omega) \right)$ depending on $q$, $\parallel f \parallel_{L^q(\Omega)}$, $\parallel \phi \parallel_{L^\infty(\partial \Omega)}$ and $diam (\Omega)$ such that any solution $u$ of the Dirichlet problem  \begin{center} $\begin{cases} u \in \qpsh(\Omega) \cap C(\overline{\Omega}) \\ (\partial \jpar u)^n = f \Omega_{n} \\ u_{|\partial \Omega}=\phi \in C(\partial \Omega) \end{cases}$ \end{center} for $f \in L^q(\Omega)$ and $q>2$, satisfies $\parallel u \parallel_{L^\infty(\Omega)} \leq C$.
\end{theorem}
\begin{proof}
Take any $s>0$, $t \in [0,1]$ and $w \in \qpsh(\Omega)$ such that $0 \leq w \leq 1$. Then 
\begin{center} $t^n\int\limits_{U(s+t)} (\partial \jpar w)^n = \int\limits_{U(s+t)} \left( \partial \jpar (tw-t-s) \right)^n = \int\limits_{\{u < v-s-t \}} \left( \partial \jpar (tw-t-s) \right)^n \leq \int\limits_{\{u < v-s + tw-t \}} \left( \partial \jpar (tw-t-s) \right)^n \leq \int\limits_{\{u < v-s + tw-t \}} \left( \partial \jpar (v + tw-t-s) \right)^n \leq \int\limits_{\{u < v-s + tw-t \}} (\partial \jpar u)^n \leq \int\limits_{\{u < v-s \}} (\partial \jpar u)^n = \int\limits_{U(s)} (\partial \jpar u)^n$ \end{center}
due to inclusions of appropriate sets, superadditivity and the comparison principle. To conclude $$t^n(b(s+t))^n \leq \int\limits_{U(s)} (\partial \jpar u)^n.$$ 

Estimating the right hand side gives 
\begin{center} $\int\limits_{U(s)} (\partial \jpar u)^n = \int\limits_{U(s)} f \Omega_{n} \leq \: \parallel f \parallel_{L^q(\Omega)} \left( \int\limits_{U(s)} 1 d \leb^{4n} \right)^\frac{1}{q'} \leq \parallel f \parallel_{L^q(\Omega)} C(p, diam(\Omega)) \left( cap(U(s),\Omega) \right)^\frac{p}{q'}= \parallel f \parallel_{L^q(\Omega)} C(p, diam (\Omega)) (b(s))^{n(1+\alpha(q))}$ \end{center} 
where we used H\"older's inequality and Lemma \ref{c-v}, $p$ depends only on $q'$ which is the conjugate of $q$ and we choose it so that $\frac{p}{q'} > 1$. 
This reassembles to $$tb(s+t) \leq A \left(q,\parallel f \parallel_{L^q(\Omega)},diam(\Omega) \right)(b(s))^{1+\alpha(q)}$$ for any $s>0$ and $t \in [0,1]$. 

We would like to apply the De Giorgi lemma, stated below, for the function $b$. Let us just note that the condition $(a)$ from the lemma is satisfied since for $s_n \searrow s$ the sets $U(s_n) \nearrow U(s)$ and under such an assumption $cap(U(s_n), \Omega) \rightarrow cap(U(s), \Omega)$, cf. \cite{WK17}. The condition $(b)$ follows from the first assertion of Lemma \ref{bded} as well as dependence of $s_0$ only on $q$, $\parallel f \parallel_{L^q(\Omega)}$ and $diam(\Omega)$. Indeed, it was proven there that
\begin{center}
$b^{ \alpha (q)} (s) = cap(U(s),\Omega)^{\frac {\alpha (q)} {n}} \leq \frac{\left( \int_\Omega (\partial \jpar u)^n \right)^{\frac {\alpha (q)} {n} }}{s^{\alpha (q)}} = \frac{ \parallel f \parallel_{L^1(\Omega)}^{\frac {\alpha (q)} {n} }}{s^{\alpha (q)}} \leq \frac{c \left( q,\parallel f \parallel_{L^q(\Omega)},diam(\Omega) \right)}{s^{\alpha(q)}}$
\end{center} 
so surely $s_0 \leq \left(  2 A c \right)^{\frac{1}{ \alpha(q) }}$ and this estimate depends only on $q$, $\parallel f \parallel_{L^q(\Omega)}$ and $diam(\Omega)$. By the De Giorgi lemma there exists $S \left( q,\parallel f \parallel_{L^q(\Omega)}, diam(\Omega) \right)$ such that $b(s)=0$ for any $s>S \left(q,\parallel f \parallel_{L^q(\Omega)}, diam(\Omega) \right)$. This together with Lemma \ref{c-v} gives our claim since then 
\begin{center} $\parallel u \parallel_{L^\infty} \leq sup |\phi| + S \left(q,\parallel f \parallel_{L^q(\Omega)}, diam(\Omega) \right)= C \left(q,\parallel f \parallel_{L^q(\Omega)},\parallel \phi \parallel_{L^\infty(\partial \Omega)}, diam(\Omega) \right)$.
\end{center}
\end{proof}

\begin{lemma}{\cite{PSS12}}[De Giorgi]
Let $f : R_+ \rightarrow R_+$ satisfy the following conditions: \\
(a) $f$ is right-continuous;\\
(b) $f$ decreases to $0$;\\
(c) There exist positive constants $\alpha, A_\alpha$ so that for all $s \geq  0$ and all $0 \leq r \leq 1$, we have
$$r f(s + r) \leq A_\alpha f(s)^{1+\alpha}.$$ 
Then there exists $s_\infty$, depending only on $\alpha, A_\alpha$ and the smallest value $s_0$ for which we have
$f(s_0)^\alpha \leq (2A_\alpha)^{-1}$ so that $f(s) = 0$ for $s > s_\infty$. In fact, we can take $s_\infty  = s_0 + 2A_\alpha(1 -
2^{-\alpha})^{-1}f(s_0)^\alpha$ .
\end{lemma}
The $L^\infty$ estimate allows us to prove stability of solutions to the Dirichlet problem in terms of densities and boundary values. This will be needed for the proof of solvability of the Dirichlet problem but is of course a result interesting in its own right. As we were told by S. Dinew the idea of proving stability presented in Proposition \ref{stab} is due to N. C. Nguyen.
\begin{lemma} \label{0case}
There exists a constant $C(q,diam(\Omega))$ depending on $q$ and $diam(\Omega)$ such that any solution $u$ of the Dirichlet problem  \begin{center} $\begin{cases} u \in \qpsh(\Omega) \cap C(\overline{\Omega}) \\ (\partial \jpar u)^n = f \Omega_{n} \\ u_{|\partial \Omega}=0 \end{cases}$ \end{center} for $f \in L^q(\Omega)$ and $q>2$, satisfy $\parallel u \parallel_{L^\infty(\Omega)} \leq C(q, diam(\Omega)) \parallel f \parallel_{L^q(\Omega)}^\frac{1}{n}$.
\end{lemma}
\begin{proof} Suppose that $\parallel f \parallel_{L^q(\Omega)} \not = 0$ otherwise there is nothing to prove. The function \begin{center} $v :=\frac{u}{\parallel f \parallel_{L^q(\Omega)}^\frac{1}{n}}$ \end{center} solve the Dirichlet problem \begin{center} $\begin{cases} v \in \qpsh(\Omega) \cap C(\overline{\Omega}) \\ (\partial \jpar v)^n = \frac{f}{\parallel f \parallel_{L^q(\Omega)}} \Omega_{n} \\ v_{|\partial \Omega}=0 \end{cases}$. \end{center} By Theorem \ref{inf} there exists a constant $C(q,diam(\Omega)):=C(q,1,0,diam(\Omega))$ such that $\parallel v \parallel_{L^\infty(\Omega)} \leq C(q, diam(\Omega))$, this gives the claim.
\end{proof}

\begin{proposition} \label{stab} There exists a constant $C(q,diam(\Omega))$ such that if $u$ and $v$ satisfy \begin{center} $\begin{cases} u \in \qpsh(\Omega) \cap C(\overline{\Omega}) \\ (\partial \jpar u)^n = f \Omega_{n} \\ u_{|\partial \Omega}=\phi \in C(\partial \Omega) \end{cases}$ and  $\begin{cases} v \in \qpsh(\Omega) \cap C(\overline{\Omega}) \\ (\partial \jpar v)^n = g \Omega_{n} \\ v_{|\partial \Omega}=\psi \in C(\partial \Omega) \end{cases}$\end{center} for $f,g \in L^q(\Omega)$, $q>2$ then \begin{center} $\parallel u-v \parallel_{L^\infty(\Omega)} \leq \sup\limits_{\partial \Omega}|\phi-\psi| + C(q,diam(\Omega)) \parallel f-g \parallel^\frac{1}{n}_{L^q(\Omega)} $. \end{center}
\end{proposition}
\begin{proof}
Consider a function $w$ being the solution of \begin{center} $\begin{cases} w \in \qpsh(\Omega) \cap C(\overline{\Omega}) \\ (\partial \jpar w)^n = (f-g)_+ \Omega_{n} \\ w_{|\partial \Omega}=0\end{cases}$.\end{center} Note that on $\partial \Omega$ we have $w+v+ \inf (\phi - \psi) \leq u$ while \begin{center} $ \left( \partial \jpar \left(w+v+ \inf (\phi - \psi) \right) \right)^n \geq (f-g)_+ + g \geq f = (\partial \jpar u)^n$. \end{center} From the comparison principle $w+v+ \inf (\phi - \psi) \leq u$ in $\overline{\Omega}$ which by Lemma \ref{0case} results in \begin{center} $u-v \geq w + \inf (\phi - \psi) \geq - C(q,diam(\Omega)) \parallel (f-g)_+ \parallel^\frac{1}{n}_{L^q(\Omega)} - \sup |\phi - \psi| \geq - C(q,diam(\Omega)) {\parallel f-g \parallel^\frac{1}{n}_{L^q(\Omega)}} - \sup |\phi - \psi|$. \end{center} The same reasoning gives \begin{center} $v-u \geq - C(q,diam(\Omega)) \parallel f-g \parallel^\frac{1}{n}_{L^q(\Omega)} - \sup |\phi - \psi|$. \end{center} This reassembles to our claim.
\end{proof}
\begin{remark} Equicontinuouity of a family of functions \begin{center} $\mathcal{P}(q,c_0,\phi)=\{u \in \qpsh(\Omega) \cap C(\overline{\Omega}) \: | \: (\partial \jpar u)^n \in L^q(\Omega), \: \int_\Omega (\partial \jpar u)^n \leq c_0, \: u_{| \partial \Omega}=\phi\}$ \end{center} for a quaternionic strictly pseudoconvex domain $\Omega$ , $q > 2$, $c_0>0$ and $\phi \in C(\partial \Omega)$ follows easily from Proposition \ref{stab}. In the complex case it was proven in \cite{K02}. \end{remark} 
\begin{theorem}
The Dirichlet problem \begin{center} $\begin{cases} u \in \qpsh(\Omega) \cap C(\overline{\Omega}) \\ (\partial \jpar u)^n = f \Omega_{n} \\ u_{|\partial \Omega}=\phi \in C(\partial \Omega) \end{cases}$ \end{center} in a smoothly bounded, quaternionic strictly pseudoconvex domain $\Omega$ for $f \in L^q(\Omega)$, $q>2$ has a unique solution.
\end{theorem}
\begin{proof}
Uniqueness follows from the comparison principle. For solvability we take a sequence of continuous non negative functions $f_i$ converging to $f$ in $L^q(\Omega)$. Solving Dirichlet problems for them with our boundary condition, which is possible due to \cite{A03b}, gives a sequence of continuous solutions $u_i$. Since, by Proposition \ref{stab}, these solutions constitute a Cauchy sequence it follows that $u_i$ converge uniformly to some $u$. This is the solution we were looking for because of convergence of Monge-Amp\`ere masses, cf. \cite{WW17}. 
\end{proof}
The example below shows that the exponent two is optimal in the sense that for densities in $L^p(\Omega)$ with $p<2$ solutions may not even be bounded. 
\begin{proposition}
Let $f(q)=\log \left( \parallel q \parallel \right)$, it belongs to $\qpsh(\hn)$ and $$(\partial \jpar f)^n=\frac{n!}{2 \parallel q \parallel^{2n}} \Omega_{n}.$$
\end{proposition}
\begin{proof} We compute for $f_\epsilon(q)= \frac{1}{2}\log \left( \parallel q \parallel^2 + \epsilon \right)$
\begin{center} $\partial \jpar f_\epsilon = \partial \left( \sum\limits_{i=0}^{2n-1} (-1)^{i+1} (\partial_{\overline{z_{i+(-1)^i}}}f_\epsilon) dz_i \right)=
\frac{1}{2}\partial \left(\sum\limits_{i=0}^{2n-1} (-1)^{i+1} \frac{z_{i+(-1)^i}}{\parallel q \parallel^2 + \epsilon} dz_i \right)= 
\frac{1}{2} \left(\sum\limits_{i=0,j=0}^{2n-1} (-1)^{i+1} \frac{\delta^j_{i+(-1)^i}(\parallel q \parallel^2 + \epsilon) -z_{i+(-1)^i} \overline{z_j}}{(\parallel q \parallel^2+\epsilon)^2} dz_j \wedge dz_i \right)=
\frac{1}{2} \left(\sum\limits_{i>j}^{2n-1} \left((-1)^{i+1} \frac{\delta^j_{i+(-1)^i}(\parallel q \parallel^2 + \epsilon) -z_{i+(-1)^i} \overline{z_j}}{(\parallel q \parallel^2+\epsilon)^2} - (-1)^{j+1} \frac{\delta^i_{j+(-1)^j}(\parallel q \parallel^2 + \epsilon) -z_{j+(-1)^j} \overline{z_i}}{(\parallel q \parallel^2 + \epsilon)^2} \right) dz_j \wedge dz_i \right)=
\frac{1}{2} \left(\sum\limits_{i>j}^{2n-1} \left( \frac{2 \delta^j_{i+(-1)^i}(\parallel q \parallel^2+\epsilon) + (-1)^{i} z_{i+(-1)^i} \overline{z_j} + (-1)^{j+1} z_{j+(-1)^j} \overline{z_i}}{(\parallel q \parallel^2 + \epsilon)^2} \right) dz_j \wedge dz_i \right).$
\end{center}
Let us denote by $M_{ij}= (-1)^{i} z_{i+(-1)^i} \overline{z_j} + (-1)^{j+1} z_{j+(-1)^j} \overline{z_i}$ as in \cite{WW17} and let $\delta^{j_1 i_1,...,j_n i_n}_{0,...,2n-1}$ be the sign of the permutation $(j_1, i_1,...,j_n, i_n) \rightarrow (0,1,...,2n-1)$. With this notation we see that

\begin{center} $2^n (\parallel q \parallel^2 + \epsilon)^{2n}(\partial \jpar f_\epsilon)^n=
\left( \sum\limits_{\substack{ j_1,i_1,...,j_n,i_n: \\ \{j_1,i_1,...,j_n,i_n\} = \{0,...,2n-1 \} \\ i_l>j_l, \: l \in \{1,...,n\} }} \left( \delta^{j_1 i_1,...,j_n i_n}_{0,...,2n-1} \prod\limits_{l \in \{1,...,n \}} \left(2 \delta^{j_l}_{i_l+(-1)^{i_l}} (\parallel q \parallel^2+\epsilon) + M_{i_l j_l} \right) \right) \right) \Omega_{n} =
{n \choose 0} \sum\limits_{\{k_1,...,k_n\}=\{0,...,n-1\}} \delta^{(2k_1) (2k_1+1),...,(2k_n) (2k_n+1)}_{0,...,2n-1} 2^n (\parallel q \parallel^2+\epsilon)^{n} +$ 
$ {n \choose 1} \sum\limits_{\substack{\{j_1,i_1,2k_2,2k_2+1,...,2k_n,2k_n+1\} = \{0,...,2n-1\} \\ i_1>j_1 \\ k_l \in \{0,...,n-1\}}} \delta^{j_1 i_1,...,(2k_n) (2k_n+1)}_{0,...,2n-1} 2^{n-1} (\parallel q \parallel^2+\epsilon)^{(n-1)} M_{i_1 j_1} +$
$ {n \choose 2} \sum\limits_{\substack{\{j_1,i_1,j_2,i_2,...,2k_n,2k_n+1\} = \{0,...,2n-1\} \\ i_1>j_1, i_2>j_2 \\ k_l \in \{0,...,n-1\}}} \delta^{j_1 i_1,j_2,i_2,...,(2k_n) (2k_n+1)}_{0,...,2n-1} 2^{n-2} (\parallel q \parallel^2+\epsilon)^{(n-2)} M_{i_1 j_1}M_{i_2 j_2} +$ \\
$ ... $ $+$ \\
$ {n \choose n} \sum\limits_{\substack{\{j_1,i_1,...,j_n,i_n\} = \{0,...,2n-1 \} \\ i_l>j_l, \: l \in \{1,...,n\}}} \delta^{j_1 i_1,...,j_n i_n}_{0,...,2n-1} \prod\limits_{l \in \{1,...,n \}} M_{i_l j_l}$. \end{center} Note that for a fixed indexes $j_3,i_3,...,j_n,i_n$ the expression \begin{center} $M'_{j_3,i_3,...,j_n,i_n}=\sum\limits_{\substack{j_1,i_1,j_2,i_2: \\ \{j_1,i_1,j_2,i_2,...,2k_n,2k_n+1\} = \{0,...,2n-1\} \\ i_1>j_1, i_2>j_2 }} \delta^{j_1 i_1,...,j_n i_n}_{0,...,2n-1} M_{i_1 j_1}M_{i_2 j_2}$ \end{center} vanish, this was already noticed in \cite{WW17}. To see this let $\{0,...,2n-1\} \setminus \{j_3,i_3,...,j_n,i_n\}=\{k,l,m,n\}$ and $k>l>m>n$. Then 

\begin{center} $\frac{1}{2} M'_{j_3,i_3,...,j_n,i_n} = \delta^{ l k, nm, j_3 i_3,...,j_n i_n}_{0,...,2n-1} M_{kl}M_{mn} + \delta^{ mk, nl, j_3 i_3,...,j_n i_n}_{0,...,2n-1} M_{km}M_{ln} + \delta^{ nk, ml, j_3 i_3,...,j_n i_n}_{0,...,2n-1} M_{kn}M_{lm}= 
\delta^{ l k, nm, j_3 i_3,...,j_n i_n}_{0,...,2n-1} \left( M_{kl}M_{mn} - M_{km}M_{ln} +  M_{kn}M_{lm} \right) =$ 
$\pm ( 
((-1)^{k} z_{k+(-1)^k} \overline{z_l} + (-1)^{l+1} z_{l+(-1)^l} \overline{z_k} )     ( (-1)^{m} z_{m+(-1)^m} \overline{z_n} + (-1)^{n+1} z_{n+(-1)^n} \overline{z_m} ) 
-((-1)^{k} z_{k+(-1)^k} \overline{z_m} + (-1)^{m+1} z_{m+(-1)^m} \overline{z_k})     ((-1)^{l} z_{l+(-1)^l} \overline{z_n} + (-1)^{n+1} z_{n+(-1)^n} \overline{z_l}) 
+((-1)^{k} z_{k+(-1)^k} \overline{z_n} + (-1)^{n+1} z_{n+(-1)^n} \overline{z_k})      ((-1)^{l} z_{l+(-1)^l} \overline{z_m} + (-1)^{m+1} z_{m+(-1)^m} \overline{z_l}) 
)=
\pm (
{(-1)^{k+m+1} z_{k+(-1)^k} \overline{z_n} z_{m+(-1)^m} \overline{z_l}} + {(-1)^{k+m} z_{k+(-1)^k} \overline{z_l} z_{m+(-1)^m} \overline{z_n}} + 
{(-1)^{l+m+1} z_{l+(-1)^l} \overline{z_k} z_{m+(-1)^m} \overline{z_n}}  + {(-1)^{m+l} z_{m+(-1)^m} \overline{z_k} z_{l+(-1)^l} \overline{z_n}} +
(-1)^{k+l+1} z_{k+(-1)^k} \overline{z_m} z_{l+(-1)^l} \overline{z_n} +  (-1)^{k+l} z_{k+(-1)^k} \overline{z_n}z_{l+(-1)^l} \overline{z_m} +  
{(-1)^{k+n+1} z_{k+(-1)^k} \overline{z_l} z_{n+(-1)^n} \overline{z_m}} + {(-1)^{k+n} z_{k+(-1)^k} \overline{z_m} z_{n+(-1)^n} \overline{z_l}} + 
{(-1)^{m+n+1} z_{m+(-1)^m} \overline{z_k} z_{n+(-1)^n} \overline{z_l}} +  {(-1)^{n+m} z_{n+(-1)^n} \overline{z_k} z_{m+(-1)^m} \overline{z_l}} + 
{(-1)^{n+l+1} z_{n+(-1)^n} \overline{z_k} z_{l+(-1)^l} \overline{z_m}} + {(-1)^{l+n} z_{l+(-1)^l} \overline{z_k} z_{n+(-1)^n} \overline{z_m}} ) = 0$.  \end{center}
Because of that only the first two summands of the expression for $(\partial \jpar f)^n$ do not vanish. We are left with
\begin{center} $(\partial \jpar f_\epsilon)^n = \frac{1}{2^n (\parallel q \parallel^2+\epsilon)^{2n}} \left( n! 2^n (\parallel q \parallel^2+\epsilon)^{n} - n! 2^{n-1} (\parallel q \parallel^2+\epsilon)^{(n-1)} \parallel q 
\parallel^2 \right) \Omega_{n} = \frac{n!(\parallel q \parallel^2 + 2 \epsilon)}{2 (\parallel q \parallel^2+ \epsilon )^{n+1}} \Omega_{n}$.\end{center} Finally since measures $(\partial \jpar f_\epsilon)^n$ converge weakly to 
$(\partial \jpar f)^n$, cf. \cite{WW17}, it is enough to find the weak limit of $\frac{n!(\parallel q \parallel^2 + 2\epsilon)}{2 (\parallel q \parallel^2+\epsilon)^{n+1}}$ which by $\frac{n!(\parallel q \parallel^2 + 
2\epsilon)}{2 (\parallel q \parallel^2+\epsilon)^{n+1}} \leq \frac{n!}{2 \parallel q \parallel^{2n}}$ and Lebesgue's Dominated Convergence theorem is $\frac{n!}{2 \parallel q \parallel^{2n}}$, exactly as we wanted.
\end{proof}

\noindent FACULTY OF MATHEMATICS AND COMPUTER SCIENCE\\
OF JAGIELLONIAN UNIVERSITY \\
\L OJASIEWICZA  6 \\
30-348, KRAK\'OW \\
POLAND \\
\textit{E-mail address:} Marcin.sroka@im.uj.edu.pl

\end{document}